\documentclass[12pt]{amsart}
\usepackage{amssymb, amsthm, epsfig, verbatim, amscd, enumerate, stmaryrd, doublespace,
amsmath}

\newtheorem{thm}{Theorem}[section]
\newtheorem{cor}[thm]{Corollary}
\newtheorem{lem}[thm]{Lemma}
\newtheorem{prop}[thm]{Proposition}

\theoremstyle{definition}
\newtheorem{defn}[thm]{Definition}

\newtheorem{rem}[thm]{Remark}

\theoremstyle{remark}

\numberwithin{equation}{section}

\newcommand{\al}{\alpha}
\newcommand{\be}{\beta}
\newcommand{\ga}{\gamma}

\newcommand{\x}{\times}

\newcommand{\Z}{\mathbb Z}

\newcommand{\R}{\mathbb R}

\newcommand{\co}{\colon\thinspace}

\begin{document}
\mathsurround=1pt 
\title[Double  curves of immersed spheres]
{Interlacement of double curves of immersed  spheres}



\author{Boldizs\'ar Kalm\'ar}
\address{Alfr\'ed R\'enyi Institute of Mathematics,
Hungarian Academy of Sciences \newline
Re\'altanoda u. 13-15, 1053 Budapest, Hungary}
\email{kalmar.boldizsar@renyi.mta.hu}

\thanks{\today}

\begin{abstract}
We characterize those unions of 
embedded disjoint circles in the sphere $S^2$ which can be the multiple point set of a generic immersion of $S^2$ into $\R^3$
in terms of the  
interlacement  of the given circles. 
Our result  is the one higher dimensional analogue of Rosenstiehl's characterization of words being Gauss codes of self-crossing plane curves.
Our proof uses a result of Lippner \cite{Li04} and we further 
generalize the ideas of Fraysseix and Ossona de Mendez \cite{FO99}, 
which leads us to directed interlacement graphs of paired trees and their 
local complementation.
\end{abstract}

\maketitle

\begin{spacing}{1.3}

\section{Introduction}


Our goal is to understand the multiple point sets of ``nicely'' mapped spheres into $3$-dimensional Euclidean space.
As explained in this section below, our task in this paper will be to manipulate certain interlaced systems of smoothly embedded circles in $S^2$. 
Although our main results are  topological-combinatorial  (see Definition~\ref{circlerel}, Remark~\ref{interremark} (1)  and Theorem~\ref{mainthm}), we
take a combinatorial-algebraic viewpoint during the proof in order to make some
analogies with \cite{FO99} transparent.
We hope readers interested in either  combinatory or topology will benefit from this approach.
To get motivated, we list two problems: first a topological  and then a combinatorial one.

\begin{enumerate}
\item
Take a $2$-knot, i.e.\  a smooth embedding  $e \co S^2 \to \R^4$, and take the map $\pi \circ e$, where
$\pi \co \R^4 \to \R^3$ is the natural projection.
  It is well-known that $\pi \circ e$ is a stable map after a small perturbation of $e$,
  moreover $\pi \circ e$ becomes a generic immersion of $S^2$ into $\R^3$ after 
  an appropriate isotopy of $e$.
  The multiple points of this generic immersion form an immersed curve in $S^2$ and one can ask 
  what is the criterion for  immersed curves in $S^2$ to come up in this way. This question is the 
 one dimension higher analogue of the problem of  chord diagrams representing knot diagrams \cite{FO99, LM76, RR78}.
In the $2$-knot case, restricting ourselves to  curves in $S^2$ which are disjoint unions of embedded circles
is related to asking about ribbon $2$-knots and their possible double decker sets, see, for example \cite{CKS04}.
So our  problem to study  is whether  a {\it virtual} ribbon $2$-knot  (i.e.\  a
disjoint union of embedded circle pairs in $S^2$) is {\it classical} (i.e.\ comes from a $2$-knot diagram).
\item
Take a circle graph (the intersection graph of chords of the unit circle in the plane), it is an undirected graph.
It is well-known that local complementation at its vertices corresponds to 
switching at the chords (for details, see, for example  \cite{Bou94, FO99} and also \cite{De36}). One can ask 
what is the corresponding class of {\it directed} graphs, i.e.\  whether 
performing a directed version of local complementation at a vertex of a directed graph $G$ 
is related to some switching-like operation
on an object such that $G$ can be interpreted as the  interlacement graph of this object.
In the present paper we show that 
the directed interlacement graph of a paired tree (defined appropriately) has this property.
\end{enumerate}
  
\subsection*{Paired trees and realization}
A disjoint union of smoothly embedded circles $$C_1, \ldots, C_{n-1}$$ in $S^2$ determines a tree, where the  vertices $v_1, \ldots, v_n$ 
of the tree
correspond to the connected components $A_1, \ldots, A_n$ of $S^2 - (C_1 \cup \cdots \cup C_{n-1})$ and
the edges $e_1, \ldots, e_{n-1}$ 
of the tree represent the circles themselves along which $A_1, \ldots, A_n$ are attached together. In a similar fashion
any tree gives us a disjoint union of embedded circles in $S^2$, which is unique up
 to  self-diffeomorphisms of $S^2$. 
  We consider the embedded circles up to such self-diffeomorphisms of $S^2$ and hence we have a bijection between trees and disjoint unions of embedded circles. We say that a tree is {\it paired} if it has an even number of edges, which are arranged into pairs. 
A paired tree or the union of the corresponding paired circles is called {\it realizable} if
there is a generic immersion of $S^2$ into $\R^3$ whose multiple point set is equal to 
the  paired circles (so the generic immersion has no triple points and we assume that two circles in $S^2$ form a pair if and only if their images are equal as sets under the immersion, for example, see Figure~\ref{immerzio}).

\begin{figure}[ht] 
\begin{center} 
\epsfig{file=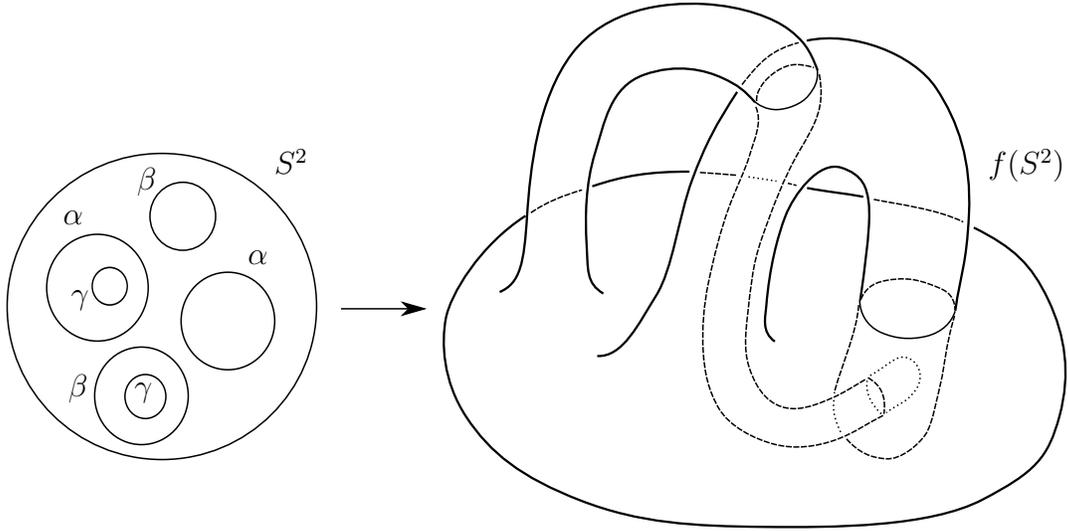, height=7cm} 
\put(-30, 134){$f(S^2)$}
\put(-300, 134){$S^2$}
\put( -380, 115){$\al$}
\put( -310, 100){$\al$}
\put( -378, 50){$\be$}
\put( -352, 128){$\be$}
\put( -353, 50){$\ga$}
\put( -377, 85){$\ga$}
\end{center} 
\caption{{A generic immersion $f \co S^2 \to \R^3$ without triple points. On the left hand side we can see
the disjoint union of the embedded circle pairs $\al$, $\be$ and $\ga$ in $S^2$, which are the double curves of $f$.}}  
\label{immerzio}  
\end{figure}

In the present paper, we study when  paired trees 
(or equivalently  disjoint unions of pairs of embedded circles in $S^2$)
are realizable. 
We expect conditions which are phrased in terms of the relative positioning of the circle pairs in $S^2$.
In the simplest way this would lead to a binary relation on the set of the given circle pairs
measuring how  they are interlaced. 

\subsection*{Interlacement of circle pairs}

 Let $k \geq 1$ be fixed and $C \subset  S^2$, $C = C_1 \cup \cdots \cup C_{2k}$ be a
disjoint union of $2k$ embedded circles.
We suppose that the components of $C$ are arranged into the pairs
${\mathbf c}_j = \{ C_j, C_{j+k} \}$, $1 \leq j \leq k$. 
We define a binary relation $\Subset$ on the set of circle pairs 
$\{  {\mathbf c}_1, \ldots, {\mathbf c}_k \}$
as follows.

\begin{defn}[Interlacement]\label{circlerel}
Let $1 \leq i \leq k$ and let $A_i \subset S^2$ denote the annulus bounded by $C_i \cup C_{i+k}$.
Then for a $1 \leq j \leq k$ the relation
 ${\mathbf c}_j \Subset {\mathbf c}_i$ holds if 
and only if
$i \neq j$ and 
there is a smooth generic embedded arc $\al$ in $A_i$ connecting $C_i$ and $C_{i+k}$ such that 
the number of intersection points 
$\al \cap (C_j \cup C_{j+k})$ is odd.
\end{defn}

\begin{rem}\label{interremark}\noindent
\begin{enumerate}
\item
${\mathbf c}_j \Subset {\mathbf c}_i$
 if and only if exactly one of $C_j, C_{j+k}$ represents the generator of 
the homology group $H_1(A_i; \Z_2)$.
\item
The relation $\Subset$ is irreflexive and not necessarily transitive or symmetric. 
\end{enumerate}
\end{rem}

%


The simplest way to visualize and work with such a relation is to form the {\it interlacement graph} of the circle pairs.

\begin{defn}[Interlacement graph]\label{linkinggraph}
The {\it  interlacement graph} $G$ of $C$  is defined as follows.
The vertices $v_1, \ldots, v_k$ of $G$ correspond to
the circle pairs 
$\mathbf c_1, \ldots, \mathbf c_k$, respectively, and
 a directed edge goes 
 from $v_i$ to $v_j$ if and only if
 $  \mathbf c_i \Subset   \mathbf c_j$.
\end{defn}

It is important to note that unlike many similar interlacement relations (coming from circle graphs for example),
our interlacement relation is not necessarily symmetric and our interlacement graph is a directed graph\footnote{In \cite{FO99}
the authors call two intersecting chords of a circle {\it interlacement} and it is a symmetric relation.}.

For better understanding, we give the corresponding relation on paired trees as well.
Since disjoint unions  of circles embedded in $S^2$ and arranged into pairs 
 are in bijection with  paired trees,
the relation $\Subset$ gives a corresponding relation between the pairs of edges of a paired tree,
which we denote by $\sqsubset$.
In a paired tree $T$ let
$\{ a_i, b_i \}, \{ a_j, b_j \}$ be two
 pairs of edges.
Then 
$$\{ a_j, b_j \} \sqsubset \{ a_i, b_i \}$$
if and only if
there is a path in $T$
connecting $a_i$ to $b_i$ 
and containing exactly one of $a_j$, $b_j$, see Figure~\ref{interlacing}.

\begin{figure}[ht] 
\begin{center} 
\epsfig{file=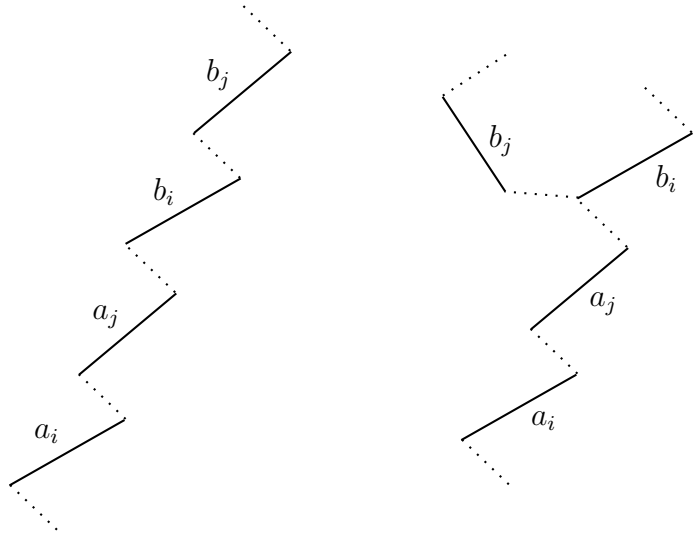, height=7cm} 
\put(-15, 130){$b_i$}
\put(-78, 145){$b_j$}
\put(-62, 40){$a_i$}
\put(-40, 85){$a_j$}
\put(-205, 125){$b_i$}
\put(-250, 35){$a_i$}
\put(-228, 80){$a_j$}
\put(-185, 170){$b_j$}
\end{center} 
\caption{{The two possible cases of $\{ a_j, b_j \} \sqsubset \{ a_i, b_i \}$. 
Note that in the picture on the left hand side also $\{ a_i, b_i \} \sqsubset \{ a_j, b_j \}$, cf.\ the two possible cases of linked pairs in \cite[Figure~6]{Li04}.
Although it may be a little confusing, our motivation to use the term `interlacement' is
to express the fact that the relations $\Subset$ and $\sqsubset$ are typically not symmetric while 
  the `linking'  in \cite{Li04} is always symmetric.}}  
\label{interlacing}  
\end{figure}

Recall that 
the {\it linking graph} $G$ of a paired tree $T$  is an {undirected} graph defined in \cite{Li04} whose vertices correspond to 
the pairs of edges of the paired tree and two vertices of $G$ are connected by an edge if and only if the
two corresponding pairs of edges $\{ a_1, b_1 \}$ and $\{ a_2, b_2 \}$ of the tree $T$
are linked, i.e.\
there exists $i, j \in \{1, 2\}$, $i \neq j$, such that 
 the unique path from the edge $a_i$ to the edge $b_i$ 
contains  $a_j$ or $b_j$ but not both, see   \cite[Figure~6]{Li04}.
Note that  our interlacement graph determines the 
linking graph of  \cite{Li04} but the converse is not true. In fact two circle pairs ${\mathbf c}_j$ and ${\mathbf c}_i$  
 are linked if and only if
${\mathbf c}_j \Subset {\mathbf c}_i$ or ${\mathbf c}_i \Subset {\mathbf c}_j$.

\subsection*{Local complementation}

A version of (directed) local complementation of directed graphs will be used later in the paper, which we define here (cf.\ local complementation in  \cite{Bou87}).
Let $G$ be a directed graph and let $p$ be a vertex of $G$.
The operation called local complementation at the vertex $p$
creates a new graph $H$ from the graph $G$ as follows.
We will say two vertices of a directed graph are connected if there is an edge (directed arbitrarily) between them.
Denote by $NG(u)$ the subgraph of $G$ whose vertices are connected (or identical) to the vertex $u$.
 
 First we add a new vertex $q$ to the graph $G$, so $H$ has one more vertex than $G$.
 The vertices of $H$ different from $q$ 
 are in natural bijection with the vertices of $G$ and we denote them with the same symbol.
 Then we have
 \begin{enumerate}[(i)]
\item
$NH(p)= NG(p)$,
\item
$NH(q) = NG(p)$,
\item
for any vertex $u$ of $G$ we have $NH(u) = NG(u)$ if $u$ 
is not connected by any edge to $p$   
in $G$ and
\item
for any two vertices $u, v$ in $H$ not equal to $p$ or $q$ if 
$u$ has an edge going into $p$ and $p$  has an edge going into $v$, then
$u$ has an edge going into $v$ in $H$ if and only if it has no such edge in $G$.
\end{enumerate}

For an example, see Figure~\ref{localcompl}.
\begin{figure}[ht] 
\begin{center} 
\epsfig{file=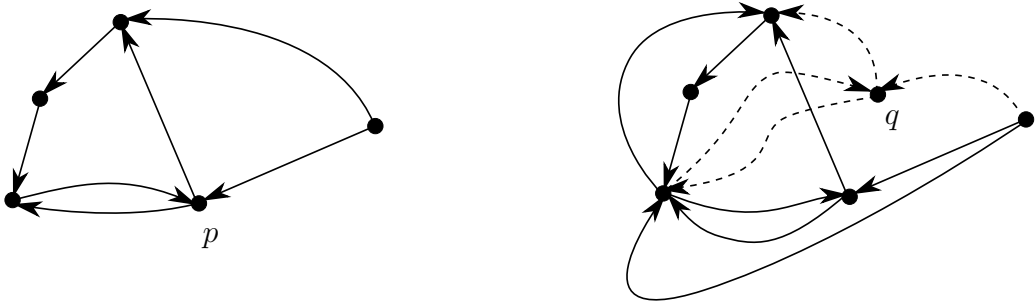, height=4cm} 
\put(-315, 22){$p$}
\put(-57, 67){$q$}
\end{center} 
\caption{{Local complementation at the vertex $p$ of the graph  on the left hand side
transforms it into the graph 
 on the right hand side. The edges starting from and going to the new vertex $q$ are drawn with different style.}}
\label{localcompl}  
\end{figure}
Note that in the new graph $H$ there is no edge between the vertices $p$ and $q$.
Local complementation has a close relationship with basis change in $\Z_2$-vector spaces and matrices, see later.

\subsection*{The double switch operation}

A characterization of realizable paired trees was given already by \cite{Li04}.
In order to characterize realizable paired trees,  \cite{Li04}  introduced  an operation called double switch 
(c.f.\ the operation D-switch in \cite{FO99}).
The double switch of a paired tree $T$ along a pair of edges $\{ a, b \}$
yields another paired tree $T'$ with one more pair of edges as follows.
Take the smallest path $P$ in $T$ containing $a$ and $b$,
denote the ending vertices  of $P$ by $\al$ and $\be$.
Then cut the tree $T$ at $\al$ and $\be$ and glue the (possibly empty) components
not containing $P$ back at the opposite places:
the component which was connected to $\al$ now is connected to $\be$ and vice versa.
Also put a new vertex $u$  in the interior of $a$ dividing it into $a'$, $a''$ and a new vertex $v$
in the interior of $b$ dividing it into $b'$, $b''$, this gives the new tree $T'$.
Obviously the edges of $T' - a' - a'' - b' - b''$\footnote{For a graph $G$ and its edge $e$ we denote by $G - e$ 
the subgraph of $G$ obtained by deleting the edge $e$.} are 
naturally identified with the edges of $T - a - b$. 
Then define the  pairing of $T'$ 
to coincide with the pairing of $T$  on the edges of $T' - a' - a'' - b' - b''$
and to be $\{a', b' \}$ and $\{ a'', b'' \}$ where we suppose that
going along a path in $T'$ containing these edges gives the order $a', a'', b'', b'$, for example, see Figure~\ref{graphdoubleswitch}.

\begin{figure}[ht] 
\begin{center} 
\epsfig{file=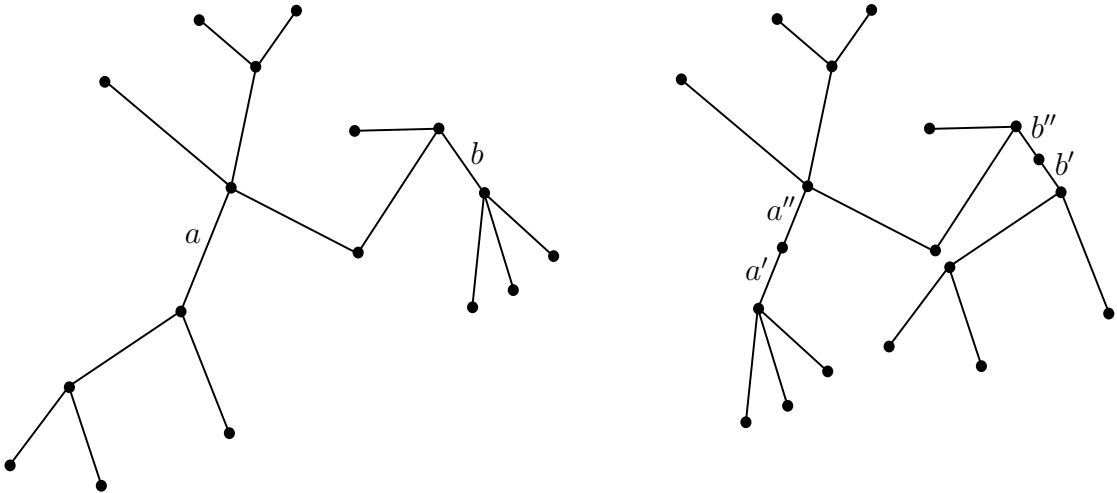, height=6.5cm} 
\put(-32, 134){$b''$}
\put(-23, 120){$b'$}
\put(-140, 80){$a'$}
\put(-132, 103){$a''$}
\put(-352, 94){$a$}
\put(-244, 124){$b$}
\end{center} 
\caption{{Double switch of a tree along the edge pair $\{a, b\}$.}}  
\label{graphdoubleswitch}  
\end{figure}

\subsection*{Characterizing realizable paired trees}

In \cite[Theorem~1]{Li04}
 a sufficient and necessary condition was obtained: 
a paired tree  is realizable if and only if after performing double switches  successively  along all of its pairs of edges 
the {linking graph} of the resulting tree is bipartite.
This characterization of paired trees 
is analogous to the characterization of words being Gauss codes in
 \cite[Theorem~6]{FO99},
 where 
 successive D-switches of linking graphs were used. 
 It
requires performing a sequence of successive
 operations (namely the double switches on the paired trees) in order to make use of it.
 From this viewpoint our main  result in the present paper
 is 
the $2$-dimensional analogue of  \cite[Theorem~10]{FO99} (see also \cite{RR78}) 
and we can say that we utilize the results of \cite{Li04}
just as Theorem~10 made practical use of Theorem~6 in \cite{FO99}.

In the present paper, we give a characterization of realizable paired trees
in terms of the interlacement graphs of the trees.
The benefit of this characterization is that
it gives an equivalent  condition for realizability in terms of an explicit property of the given paired tree.
For instance, from our Theorem~\ref{mainthm} immediately follows that
if among the given system of circle pairs in $S^2$ there is a circle pair $C_1, C_2$ such that 
an {\it odd number} of other circle pairs are positioned with $C_1, C_2$  as in Figure~\ref{interlaceexample}, then
this system is not realizable, for more details, see Corollary~\ref{egyszerukov} and Corollary~\ref{intergraphnecessary}.\
(For an example of a paired tree and applying this criterion, see Figure~\ref{interlacementgraphexample}.)

\begin{figure}[ht] 
\begin{center} 
\epsfig{file=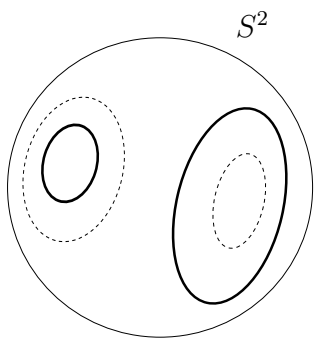, height=4cm} 
\put(-30, 114){$S^2$}
\end{center} 
\caption{{Two pairs of circles in $S^2$. The two dashed circles form a pair and the two bold circles form a pair.}}  
\label{interlaceexample}  
\end{figure}

One of our most important observations 
is that the interlacement graphs of paired trees before and after a double switch are related to each other
by a version of the  local complementation for directed graphs, we explain this fact in detail during the proof of 
Proposition~\ref{switchneigh}. 
(Although this seems to be an analogue of the local complementation in \cite[Section~2]{FO99},
it is  less obvious. Also note that  local complementation can become useful for us 
only because we have found the right notion of directed interlacement graphs.)
By this knowledge we can compare the interlacement graphs and we can 
find properties of the graphs which remain unchanged during double switches.
We formalize these properties in an algebraic way.
The difficulty at this point is that the incidence matrix of a directed graph is not necessarily symmetric and hence
we need a  more sophisticated
argument than the straightforward algebraization 
  of \cite[Section~4]{FO99}.
These and \cite{Li04} together yield our main Theorem~\ref{mainthm},
which characterizes the realizable circle pairs in $S^2$.

%

The paper is organized as follows.
In Section~2 we announce our main results.
In Section~3, we state Theorem~\ref{condbipartite}, which is of central importance and allows us to prove our main results (its proof is
deferred to later sections).
In Section~4  we describe in a very detailed way how the double switch operation on paired trees
acts on their interlacement graphs. By using this theory we prove Theorem~\ref{condbipartite}
in Section~4
with the help of some quite technical statements. Finally, these technical statements are fully proved only in Section~5.

\subsection*{Acknowledgements}

The author was supported by 
OTKA NK81203 and by the Lend\"ulet program 
of the Hungarian Academy of Sciences.
The author thanks the referees for the valuable comments which improved the paper.
 
\section{Main results}

So let $k \geq 1$ be fixed and $C \subset  S^2$, $C = C_1 \cup \cdots \cup C_{2k}$ be a
disjoint union of $2k$ embedded circles arranged into the pairs
${\mathbf c}_j = \{ C_j, C_{j+k} \}$, $1 \leq j \leq k$.

%


For any two circle pairs $\mathbf c$ and $\mathbf d$
let $[\mathbf c, \mathbf d ]$ denote
the set of circle pairs $\mathbf x$ such that $\mathbf c \Subset \mathbf x$ and $\mathbf x \Subset \mathbf d$.
Our main result is the following.
As before, let $\{  {\mathbf c}_1, \ldots, {\mathbf c}_k \}$ be a given set of disjoint circle pairs in $S^2$.
\begin{thm}[Main theorem]\label{mainthm}
The realizability of $C$ is equivalent to the following condition: 
 the set $\{ 1, \ldots, k \}$
 has a bipartition such that
 for any  $i, j \in \{ 1, \ldots, k \}$ 
 \begin{enumerate}[{\rm (i)}]
\item
$i$ and $j$ belong to the same class of the bipartition and
\item
$\mathbf c_i \Subset \mathbf c_j$
 \end{enumerate}
 if and only if 
  $[\mathbf c_i, \mathbf c_j ]$ 
has an odd  number of elements.
\end{thm}

Theorem~\ref{mainthm} and  Theorem~\ref{mainthm2} below are analogous 
to Rosenstiehl's characterization of words being Gauss codes of immersed plane curves,
see, for example, \cite[Section~4, Theorem~10]{FO99}.

By taking $i=j$ in Theorem~\ref{mainthm},  
 we get 

\begin{cor}\label{egyszerukov}
If $C$ is realizable, then for any 
circle pair $\mathbf c$
the number of circle pairs $\mathbf x$ such that 
$\mathbf c \Subset \mathbf x$ and $\mathbf x \Subset \mathbf c$ both hold is even.
\end{cor}

%

%

%

In the following, we rephrase Theorem~\ref{mainthm}
in terms of directed graphs. 
For a vertex $v$ of $G$ let $N_{out}(v)$ and $N_{in}(v)$ denote
the set of out-neighbors and the set of in-neighbors of $v$, respectively.
The following theorem is clearly equivalent to Theorem~\ref{mainthm}.

\begin{thm}\label{mainthm2}
$C$ is realizable  if and only if
 the set of vertices of $G$ has a bipartition such that
 for any two vertices $u, v$ of $G$ 
the number of vertices in  $N_{out}(u) \cap N_{in}(v)$ is odd if and only if
\begin{enumerate}[{\rm (i)}]
\item
$u$ and $v$ belong to the same class of the partition and
\item
$v \in N_{out}(u)$.
\end{enumerate}
\end{thm}
\begin{proof}
The graph $G$ just expresses the relation $\Subset$ so the statement follows immediately from Theorem~\ref{mainthm}.
\end{proof}

By taking $i=j$ in Theorem~\ref{mainthm2}, we get

\begin{cor}\label{intergraphnecessary}
In particular, if $C$ is realizable, then 
the number of neighbors of any vertex $v$ of the interlacement graph $G$ which are connected by two oppositely
directed edges to $v$ is even.
\end{cor}

For example, the paired tree in Figure~\ref{interlacementgraphexample} is not realizable.

\begin{figure}[ht] 
\begin{center} 
\epsfig{file=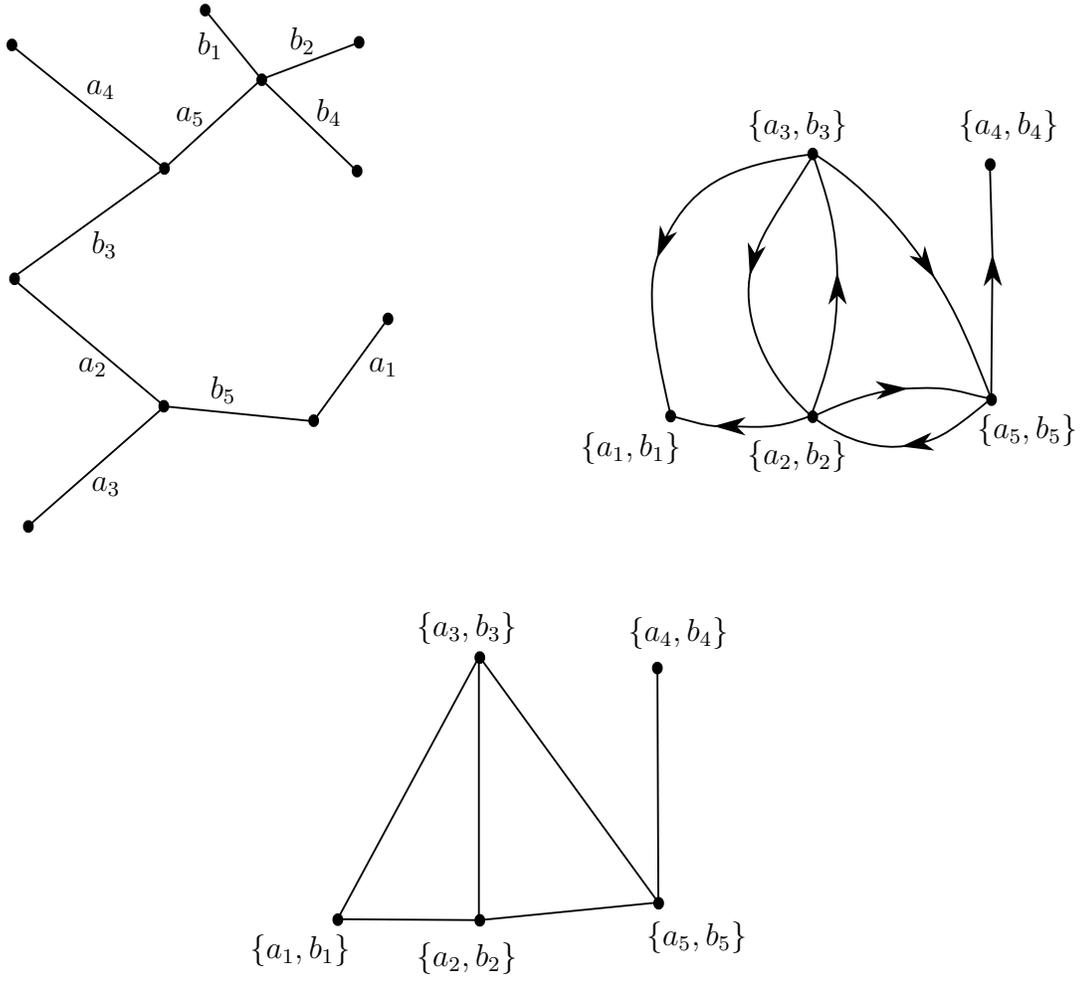, height=12.3cm} 
\put(-345, 255){$b_3$}
\put(-345, 165){$a_3$}
\put(-350, 210){$a_2$}
\put(-270, 332){$b_2$}
\put(-347, 315){$a_4$}
\put(-260, 305){$b_4$}
\put(-313, 305){$a_5$}
\put(-300, 200){$b_5$}
\put(-305, 330){$b_1$}
\put(-240, 210){$a_1$}
\put(-160, 178){$\{ a_1, b_1 \}$}
\put(-97, 175){$\{ a_2, b_2 \}$}
\put(-10, 185){$\{ a_5, b_5 \}$}
\put(-97, 300){$\{ a_3, b_3 \}$}
\put(-17, 300){$\{ a_4, b_4 \}$}
\put(-285, -12){$\{ a_1, b_1 \}$}
\put(-222, -15){$\{ a_2, b_2 \}$}
\put(-135, -7){$\{ a_5, b_5 \}$}
\put(-222, 110){$\{ a_3, b_3 \}$}
\put(-142, 108){$\{ a_4, b_4 \}$}
\end{center} 
\caption{{A paired tree, its interlacement graph on the right hand side and finally its linking graph. 
The vertices of the interlacement graph and linking graph
correspond to the edge pairs of the tree.
By Corollary~\ref{intergraphnecessary} this paired tree is not realizable as we can see immediately from its interlacement graph.}}  
\label{interlacementgraphexample}  
\end{figure}

\section{Proof of Main theorem}
 
\begin{defn}[Incidence matrix]
A directed  graph
$G$
with  vertices $p_1, \ldots, p_m$ 
 determines an $m \x m$ incidence matrix over $\Z_2$ as usual: let $L \in \Z_2^{m \x m}$ be the matrix 
such that the element $[L]_{i, j}$ in the $i$-th row and $j$-th column of $L$ is
\begin{itemize}
\item
equal to $1$ if
there is a directed edge of $G$ going from $p_j$ to $p_i$ and
\item
equal to $0$ in any other case.
\end{itemize}
\end{defn}

We have the following very important result about the incidence matrix of the interlacement graph of a paired tree.
For a vector space $\Z_2^r$ 
let $\langle \cdot, \cdot \rangle \co \Z_2^{r} \x \Z_2^{r} \to \Z_2$ denote the standard scalar product.

\begin{thm}\label{condbipartite}
Let  $L^{k \x k}$ be
the incidence matrix of  the interlacement graph $G_0$ of
a paired tree $T_0$ with edge pairs $\{ a_1, b_1 \}, \ldots, \{ a_k, b_k \}$. 
Performing double switches along all the edge pairs 
$\{ a_1, b_1 \}, \ldots, \{ a_k, b_k \}$ successively gives
the paired trees $T_1, \ldots, T_k$ and the corresponding interlacement graphs $G_1, \ldots, G_k$.
 Then
the interlacement graph $G_k$ of the  paired tree $T_k$ 
is bipartite if and only if 
there exists a vector $A \in \Z_2^{k}$ such that
$$
\langle Lu, L^Tv \rangle = \langle Lu, v \rangle ( 1 + \langle A, u+ v \rangle )
$$
holds for any standard basis vectors $u, v \in \Z_2^{k}$.
\end{thm}

The proof of this theorem will be presented later in Section~4 on page 17.
Nevertheless by using this result and \cite[Theorem~1]{Li04}, now we can prove our main theorem:
\begin{proof}[Proof of Theorem~\ref{mainthm}]
By  \cite[Theorem~1]{Li04} it is easy to see that the interlacement graph $G_k$ is bipartite if and only if the tree $T_0$ is
realizable. On the other hand notice that 
for a basis vector $u$ representing a vertex $p$ the vector $Lu$ shows 
for  which vertices $q$ there are oriented edges from $p$ to $q$.
Similarly $L^Tu$ shows the oriented edges into $p$.
So $\langle Lu, L^Tv \rangle$ counts the vertices which are out-neighbors of $u$ and in-neighbors of $v$.
Also notice that the vector $A$ accounts for the bipartition.
Hence the algebraic condition in the statement of Theorem~\ref{condbipartite} is clearly equivalent to the condition about the circle pairs in 
Theorem~\ref{mainthm}.
(Or equivalently we could say the same thing about Theorem~\ref{mainthm2}.)
\end{proof}


\section{The effect of the double switch operation and local complementation}

In this section, we prove Theorem~\ref{condbipartite} but at first
we describe in detail how double switch works.

If $T_0$ is a paired tree with edge  pairs $\{ \{ a_i, b_i \} : 1 \leq i \leq k \}$,
then its edges $a_2, b_2, \ldots, a_k, b_k$ are naturally identified with edges of the paired tree  $T_1$
 obtained by double switch along $\{ a_1, b_1 \}$ as it was explained in the Introduction. 
 We could identify $\{ a_1, b_1 \}$ with $\{ a_1', b_1' \}$ or $\{ a_1'', b_1'' \}$ in $T_1$, let us 
 fix the convention that $\{ a_1, b_1 \}$ is  identified 
 with $\{ a_1', b_1' \}$. Clearly, the number of vertices of the interlacement graph $G_1$ of $T_1$ is one more than 
 that of the interlacement graph $G_0$ of $T_0$ and
  the vertices $p_1, \ldots, p_k$ 
 of  $G_0$ are identified with the corresponding vertices  of $G_1$.
 We denote the additional vertex of $G_1$ by $q_1$,
 it corresponds to the edge pair $\{ a_1'', b_1'' \}$ and it has
   exactly the same in- and out-neighbors 
 as $p_1$. 
 Similarly double switching successively along the next pairs of edges 
introduces the new vertices $q_2, \ldots, q_k$ into the new interlacement graphs
$G_2, \ldots, G_k$. The whole process also gives the corresponding new incidence matrices $L_1, \ldots, L_k$.

Since each of the $k$ double switches  along the  edge pairs increases the number of vertices of the interlacement graph 
and hence the size of the incidence matrix  by one,
it is convenient to put all the matrix $L_i \in \Z_2^{(k+i) \x (k+i)}$ into the upper left corner of a larger $2k \x 2k$ matrix denoted  by $M_i \in \Z_2^{2k \x 2k}$ 
and declare the other elements of this larger $M_i$ to be  zero.
The diagonal elements  of $M_i$ are all zero since the interlacement graph $G_i$ has no loops.

To prove Theorem~\ref{condbipartite}, we will need some more technical results.
Starting with a paired tree $T_0$, its linking graph $G_0$ with vertices $p_1, \ldots, p_k$ and the
sequence of double switch operations along the pairs of edges $\{a_1, b_1 \}, \ldots, \{ a_k, b_k \}$
 can be seen in a ``vector space language'' as follows.

The vector space $\Z_2^{2k}$  splits as $\Z_2^{2k} = \Z_2^k \oplus \Z_2^k$ and from now on we identify 
the standard basis of the first $\Z_2^k$ summand with the vertices $p_1, \ldots, p_k$ of $G_0$.
Likewise the second $\Z_2^k$ summand is preserved for the new 
vertices $q_1, \ldots q_k$ coming from the double switches later.
So after double switching along $\{a_1, b_1\}$ the new vertex $q_1$ enters into the picture, the edges of
the new interlacement graph $G_1$ are perhaps also different from that of $G_0$, and the new incidence matrix $M_1 \in \Z_2^{2k \x 2k}$ of $G_1$ 
will have perhaps one more non-zero row and column: the $(k+1)$th row and column represents
how the new vertex $q_1$ participates in $G_1$.

This implies that after $0 \leq i$ successive double switches along 
$p_1, \ldots, p_i$, where $i \leq k$, the upper left $(k+i) \x (k+i)$ submatrix of $M_i$
tells us how the edges of the interlacement graph $G_i$ connect the vertices $p_1, \ldots, p_k, q_1, \ldots, q_i$.
Therefore if $u$ is one of the first $k+i$ standard basis elements of  
$\Z_2^{2k}$, then the vector $M_i u$ represents the out-neighbors of the vertex (represented by) $u$ in $G_i$.

The following statement describes precisely how the incidence matrix changes under double switch of a tree.

\begin{prop}\label{switchneigh}
Let  $k \geq 1$ and $0 \leq i \leq k-1$.
Let $T_i$ be a paired tree with $2(k+i)$ edges, $G_i$ be its interlacement graph with $k+i$ vertices and $M_i \in \Z_2^{2k \x 2k}$  be its
incidence matrix
as explained above.
The paired tree obtained after double switching $T_i$ along the pair of edges $\{ a_{i+1}, b_{i+1} \}$
is denoted by $T_{i+1}$, its interlacement graph is denoted by $G_{i+1}$.
If $p_{i+1}$ denotes the vertex of $G_i$ and $G_{i+1}$ corresponding to $\{ a_{i+1}, b_{i+1} \}$
and $q_{i+1}$ denotes the new vertex of $G_{i+1}$, then 
\begin{enumerate}[\rm (1)]
\item
the incidence matrix $M_{i+1} \in \Z_2^{2k \x 2k}$  of $T_{i+1}$ is obtained from $M_i$  by 
\begin{itemize}
\item
$M_{i+1}u = M_i u + \langle M_i u, p_{i+1} \rangle \left( (M_i p_{i+1} +q_{i+1}) + \langle u, M_i p_{i+1} \rangle u \right)$, if  $u$ 
is a standard basis element representing a vertex of $G_{i+1} - q_{i+1}$ and
\item
$M_{i+1} q_{i+1} = M_i p_{i+1}$, 
\end{itemize}
\item
the incidence matrix $M_{i}$  is obtained from $M_{i+1}$  by
\begin{itemize}
\item
$M_i u = M_{i+1} u + \langle M_{i+1} u, p_{i+1} \rangle \left( (M_{i+1} p_{i+1} +q_{i+1}) + \langle u, M_{i+1} p_{i+1} \rangle u \right)$, if  $u$ 
is a standard basis element representing a vertex of $G_{i}$ and
\item
$M_i q_{i+1} = 0$.
\end{itemize}
\end{enumerate}
\end{prop}

%

%


\begin{rem}
The same formulas hold for the matrices $M_i^T$ instead of $M_i$.
\end{rem}

Now we prove Proposition~\ref{switchneigh}.

\begin{proof}[Proof of Proposition~\ref{switchneigh}]
We study how the interlacement graph $G_i$ is changing during the double switch of $T_i$ along $\{a_{i+1}, b_{i+1} \}$.
We show that $G_{i+1}$ 
 is  obtained from $G_i$ by   {directed} local complementation at $p_{i+1}$,
 where we add the vertex $q_{i+1}$ with the same neighborhood as that of $p_{i+1}$
 (see in the Introduction and   cf.\ \cite[Section~2.2]{FO99}).
  More precisely, denote by $NG_i(u)$ the subgraph of $G_i$ whose vertices are connected (or identical) to $u$. We show that
\begin{enumerate}[(i)]
\item
$NG_{i+1}(p_{i+1})= NG_{i}(p_{i+1})$,
\item
$NG_{i+1}(q_{i+1}) = NG_{i}(p_{i+1})$,
\item
for any vertex $u$ of $G_i$ we have $NG_{i+1}(u) = NG_{i}(u)$ if $u$ 
is not connected by any edge to $p_{i+1}$   
in $G_i$ and
\item
for any two vertices $u, v$ in $G_{i+1}$ not equal to $p_{i+1}$ or $q_{i+1}$ if 
$u$ has an edge going into $p_{i+1}$ and $p_{i+1}$  has an edge going into $v$, then
$u$ has an edge going into $v$ in $G_{i+1}$ if and only if it has no such edge in $G_i$.
\end{enumerate}
Note that the conditions  (1) and (2) in the statement of Proposition~\ref{switchneigh}
are clearly equivalent to (i)-(iv) if $M_i$ denote the corresponding incidence matrices so
it is enough to show that (i), (ii), (iii) and (iv) hold.

If $x, y$ are two edges of a tree, then let $\overline{xy}$ denote the path connecting (and containing) $x$ and $y$.
For two subgraphs $G, H$ of a graph, let $G - H$ denote the subgraph
of $G$ which has the same vertices as $G$ and has an edge if and only if $G$ has it but $H$ does not have it.
We also  introduce some other notations:
if both of $\{ a_j, b_j \} \sqsubset \{ a_i, b_i \}$ and $\{ a_j, b_j \} \sqsupset \{ a_i, b_i \}$ hold, then 
we will write $\{ a_j, b_j \} \parallel \{ a_i, b_i \}$. If none of $\{ a_j, b_j \} \sqsubset \{ a_i, b_i \}$ and $\{ a_j, b_j \} \sqsupset \{ a_i, b_i \}$ hold, then 
we will write $\{ a_j, b_j \} \sqsupset \sqsubset  \{ a_i, b_i \}$.

\begin{figure}[ht] 
\begin{center} 
\epsfig{file=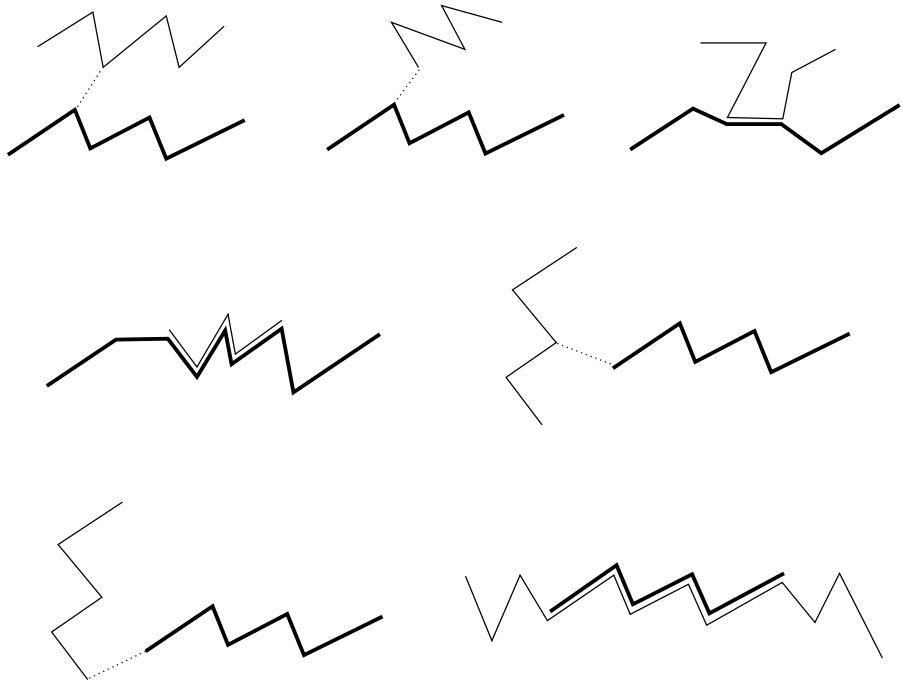, height=9cm} 
\end{center} 
\caption{{The bold  zigzag symbolizes the path $\overline{a_{i+1}  b_{i+1}}$ and 
the thin zigzag symbolizes the path   $\overline{xy}$ in the tree $T_i$.
The dotted lines symbolize any paths (or just one vertex).
In the first six cases $\overline{xy}$ lies in one of the components of $T_i  - a_{i+1} - b_{i+1}$.}}
\label{ketut}  
\end{figure}

At first let us show that (i) holds.
Take any edge pair $\{x, y \}$ of $T_i$ different from $\{a_{i+1}, b_{i+1} \}$.
Suppose $\{x, y \} \sqsubset \{a_{i+1}, b_{i+1} \}$. This means that  exactly one of $x, y$ is in 
$\overline{a_{i+1} b_{i+1}}$. Clearly this property still holds after double switch along $\{a_{i+1}, b_{i+1} \}$.
Now suppose $\{x, y \} \sqsupset \{a_{i+1}, b_{i+1} \}$. This means
that exactly one of $a_{i+1}$ and $b_{i+1}$ is in $\overline{xy}$, say $a_{i+1}$.
Then after double switch, only $b_{i+1}$ will be in $\overline{xy}$.
So if a vertex of $G_i$ is connected to $p_{i+1}$, then it will be connected to $p_{i+1}$ in $G_{i+1}$ in the same way.
It follows that  if  $\{ x, y \} \sqsupset \sqsubset \{a_{i+1}, b_{i+1} \}$ in $T_i$, then $\{ x, y \} \sqsupset \sqsubset \{a_{i+1}, b_{i+1} \}$ in $T_{i+1}$ as well
because double switching twice along an edge pair keeps the paired tree 
  (except the new dividing vertices on the edges along which we double switch)
 and so the interlacement relation.
This argument shows that (i) holds.

Now it easy to see that (ii) holds because (i) holds and the new vertex $q_{i+1}$ comes from 
dividing the edges $a_{i+1}$ and $b_{i+1}$ as we explained in the Introduction.

To show (iii), take an edge pair $\{ x, y \}$ different from  $\{a_{i+1}, b_{i+1} \}$
such that  $\{ x, y \} \sqsupset \sqsubset \{a_{i+1}, b_{i+1} \}$ in $T_i$.
There are two possibilities: 
$\overline{xy}$ lies completely in one of the components of $T_i  - a_{i+1} - b_{i+1}$, or
$\overline{xy}$ contains both of $a_{i+1}, b_{i+1}$, see Figure~\ref{ketut}.

At first suppose that 
$\overline{xy}$ lies completely in one of the components of $T_i  - a_{i+1} - b_{i+1}$.
Then clearly for any edge pair $\{ u, v \}$ such that $\{ u, v \} \sqsubset \{ x, y \}$ in  $T_i$
we have the same in  $T_{i+1}$ as well.
On the other hand, if $\{ u, v \} \sqsupset \{ x, y \}$, then 
 assume $\overline{uv}$ contains $x$ and $\overline{uv}$ does not contain $y$.
 It is easy to see that in the first three cases and in the  fifth and sixth cases of Figure~\ref{ketut}
 the same holds after double switch and in the fourth case of Figure~\ref{ketut}
 after double switch still $\overline{uv}$ contains one of $x$ or $y$.  
If we  suppose that $\overline {x y }$ contains both of $a_{i+1}, b_{i+1}$ (this is the seventh case of  Figure~\ref{ketut}), then
 the argument is similar.
 On the other hand, if $\{ x, y \} \sqsupset \sqsubset \{ u, v \}$ in $T_i$, then $\{ x, y \} \sqsupset \sqsubset \{ u, v \}$ in $T_{i+1}$ as well
 since one more double switch along the same edge pair gives back $T_i$ (up to dividing the edges $a_{i+1}, b_{i+1}$).
So we proved (iii).

\begin{figure}[ht] 
\begin{center} 
\epsfig{file=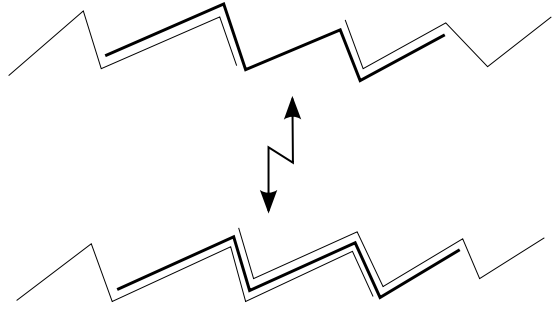, height=4cm} 
\end{center} 
\caption{The path $\overline{a_{i+1}  b_{i+1}}$ is symbolized by the bold zigzag. 
The thin zigzags symbolize the paths   $\overline{xy}$ and $\overline{uv}$ in the tree $T_i$.}
\label{haromutparhuzamos}  
\end{figure}

Now let us show (iv).
Let $\{ u, v \}$ and $\{ x, y \}$ be different edge pairs in $T_i$, also different from $\{ a_{i+1}, b_{i+1} \}$. 
There are some cases depending on how they are related to 
$\{ a_{i+1}, b_{i+1} \}$ and to each other.


\begin{enumerate}
\item
If in $T_i$ the relations  $\{ u, v \} \sqsubset \{ a_{i+1}, b_{i+1} \}$ and $\{ x, y \} \sqsubset \{ a_{i+1}, b_{i+1} \}$ both hold
but not $\{ u, v \} \sqsupset \{ a_{i+1}, b_{i+1} \}$ and not  $\{ x, y \} \sqsupset \{ a_{i+1}, b_{i+1} \}$, then
$\{ u, v \}$ and $\{ x, y \}$ lie in the component of $T_i - a_{i+1} - b_{i+1}$ containing the inner vertices of $\overline {a_{i+1} b_{i+1} }$.
In this case clearly $\{ u, v \}$ and $\{ x, y \}$ are related  in $T_{i+1}$ as they are in $T_i$.
\item
If $\{ u, v \} \sqsupset \{ a_{i+1}, b_{i+1} \}$ and $\{ x, y \} \sqsupset \{ a_{i+1}, b_{i+1} \}$ 
but not $\{ u, v \} \sqsubset \{ a_{i+1}, b_{i+1} \}$ and not  $\{ x, y \} \sqsubset \{ a_{i+1}, b_{i+1} \}$ in $T_i$, then
assume $u$ and $x$ are in the component of $T_i - a_{i+1} - b_{i+1}$ containing the inner vertices of $\overline {a_{i+1} b_{i+1} }$.
Then $v$ and $y$ are in the other components of $T_i - a_{i+1} - b_{i+1}$, which will be  interchanged during double switch.
Of course $u$ and $x$ are not in $\overline {a_{i+1} b_{i+1} }$. These imply that 
 $\{ u, v \}$ and $\{ x, y \}$ are related  in $T_{i+1}$ the same way as in $T_i$.
\item
If $\{ u, v \} \parallel \{ a_{i+1}, b_{i+1} \}$ and $\{ x, y \} \parallel \{ a_{i+1}, b_{i+1} \}$ in $T_i$, then 
$\{ u, v \} \parallel \{ x, y \}$ or $\{ u, v \}   \sqsupset  \sqsubset  \{ x, y \}$ must hold in $T_i$,
see Figure~\ref{haromutparhuzamos}.
In this case the relation between $\{ u, v \}$ and $\{ x, y \}$ in the new paired tree $T_{i+1}$
changes as we claim:
$\{ u, v \} \parallel \{ x, y \}$  (resp.\ $\{ u, v \}   \sqsupset  \sqsubset  \{ x, y \}$) in $T_{i+1}$ if and only if 
$\{ u, v \}   \sqsupset  \sqsubset  \{ x, y \}$ (resp.\  $\{ u, v \} \parallel \{ x, y \}$) in $T_i$. 
See Figure~\ref{haromutparhuzamos}.
\item
Or else, if $\{ u, v \} \parallel \{ a_{i+1}, b_{i+1} \}$ and $\{ x, y \} \sqsubset \{ a_{i+1}, b_{i+1} \}$ or 
$\{ x, y \} \sqsupset \{ a_{i+1}, b_{i+1} \}$
in $T_i$, then the possible configurations can be
seen in  Figure~\ref{haromutegyikparhuzamos}.
It is easy to check that our claim holds, details are left to the reader.
\begin{figure}[ht] 
\begin{center} 
\epsfig{file=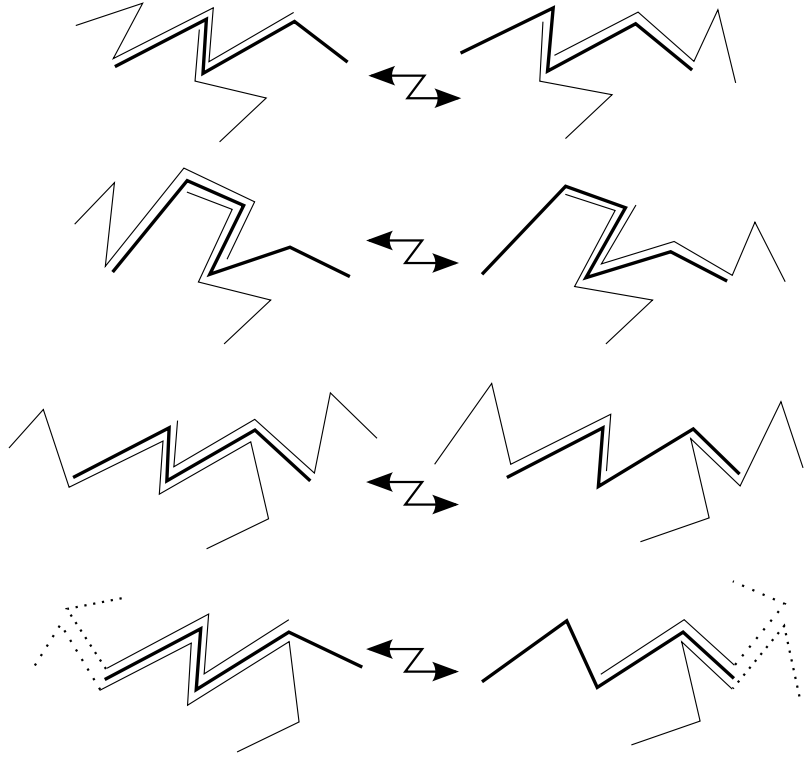, height=10cm} 
\end{center} 
\caption{The bold zigzag symbolizes the path $\overline{a_{i+1}  b_{i+1}}$ and 
the thin (possibly dotted) zigzags symbolize the paths   $\overline{xy}$ and $\overline{uv}$ in the tree $T_i$. 
$\{ u, v \} \parallel \{ a_{i+1}, b_{i+1} \}$ holds in any cases, $\{ x, y \} \sqsubset \{ a_{i+1}, b_{i+1} \}$ holds
in the first four and
$\{ x, y \} \sqsupset \{ a_{i+1}, b_{i+1} \}$ holds in the last four cases.
The relations between $\{ u, v \}$ and $\{ x, y \}$ in the first six cases are the following:
$\{ u, v \} \sqsupset \{ x, y \}$, $\{ u, v \} \sqsupset \sqsubset \{ x, y \}$,
$\{ u, v \} \parallel \{ x, y \}$, $\{ u, v \} \sqsubset \{ x, y \}$,
$\{ u, v \} \sqsubset \{ x, y \}$, $\{ u, v \} \sqsupset \sqsubset \{ x, y \}$.}
\label{haromutegyikparhuzamos}  
\end{figure}
\item
Finally, if none of the above cases hold but
$\{ u, v \} \sqsupset \{ a_{i+1}, b_{i+1} \}$ and $\{ x, y \} \sqsubset \{ a_{i+1}, b_{i+1} \}$, then 
see Figure~\ref{haromutnemparh}.
\begin{figure}[ht] 
\begin{center} 
\epsfig{file=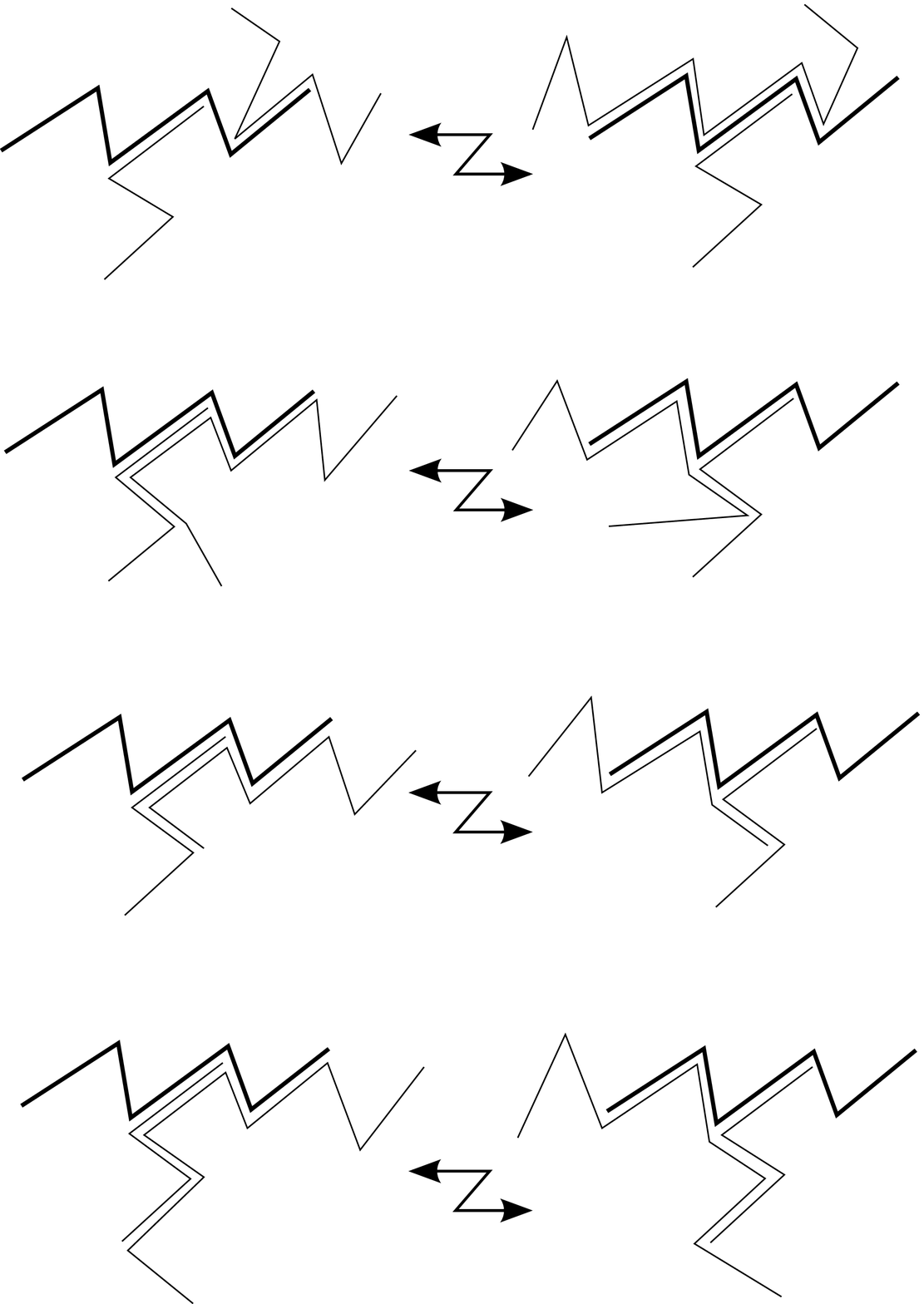, height=13cm} 
\end{center} 
\caption{The bold  zigzag symbolizes the path $\overline{a_{i+1}  b_{i+1}}$ and 
the thin zigzags symbolize the paths   $\overline{xy}$ and $\overline{uv}$ in the tree $T_i$. 
$\{ u, v \} \sqsupset \{ a_{i+1}, b_{i+1} \}$ and $\{ x, y \} \sqsubset \{ a_{i+1}, b_{i+1} \}$ hold in all cases.}
\label{haromutnemparh}  
\end{figure}
\end{enumerate}

\end{proof}

To prove Theorem~\ref{condbipartite} we will need the following important proposition. 
Let $\mathcal B (\Z_2^r)$ denote the set of standard basis elements of a vector space $\Z_2^r$.

\begin{prop}\label{inductionlemma}
Let $0 \leq i \leq k$ be fixed and 
let $A \in \Z_2^{2k}$ be a fixed vector.
If for any $x, y \in \mathcal B (\Z_2^{2k})$  we have 
$$\langle M_i x, M^T_i y \rangle
= \langle M_{i} x, y \rangle (1+ \langle A, x + y \rangle),$$
then there exist vectors $B, C \in  \Z_2^{2k}$ such that
we have for any  $u, v \in \mathcal B (\Z_2^{2k})$
\begin{enumerate}[\rm (1)]
\item
\begin{enumerate}[\rm (i)]
\item 
 the $k+i+2, \ldots, 2k$-th coordinates of $B$ are all zero and 
 \item
 $
\langle M_{i+1} u  , M^T_{i+1} v \rangle = 
\langle M_{i+1} u , v \rangle  ( 1 +  \langle B, u  + v \rangle  ) 
$
\end{enumerate}
if $0 \leq i \leq k-1$ and the $k+i+1, \ldots, 2k$-th coordinates of $A$ are all zero, and  
\item
$$
\langle M_{i-1} u  , M^T_{i-1} v \rangle = 
\langle M_{i-1} u , v \rangle  ( 1 +  \langle C, u  + v \rangle  )$$
if $1 \leq i \leq k$.
\end{enumerate}
\end{prop}
The proof of this proposition is quite technical and so it will be presented only in the next section.
Now we prove Theorem~\ref{condbipartite}.

\begin{proof}[Proof of Theorem~\ref{condbipartite}]
First suppose that 
there exists a vector $A \in \Z_2^{k}$ such that
$$
\langle Lu, L^Tv \rangle = \langle Lu, v \rangle ( 1 + \langle A, u+ v \rangle )
$$
holds for any $u, v \in \mathcal B (\Z_2^{k})$.
As explained at the beginning  of Section~4, we consider $L$ as a $2k \x 2k$ matrix $M_0$ and also 
the vector $A$ as a vector in $\Z_2^{2k}$ whose $k+1, \ldots, 2k$-th coordinates are all zero.
Then still
$$
\langle M_0 u, M_0 ^Tv \rangle = \langle M_0 u, v \rangle ( 1 + \langle A, u+ v \rangle )
$$
holds for any $u, v \in \mathcal B (\Z_2^{2k})$.
Then by applying Proposition~\ref{inductionlemma} (1) iteratively,
after performing the double switches 
the final incidence matrix $M_k$ of the tree $T_k$ satisfies the same type of equation with some vector $A_k \in \Z_2^{2k \x 2k}$ instead of $A$.
Clearly $A_k$ determines a bipartition of  the interlacement graph of the tree $T_k$ by the conditions
$\langle A_k, u \rangle = 0$ or $\langle A_k, u \rangle = 1$ for a vertex represented by $u$.

Since in the interlacement graph of the tree $T_k$ 
for every vertex the adjacent edges are duplicated due to the double switches, we have  $\langle M_k u, M_k^Tv \rangle = 0$
for any $u, v \in \mathcal B (\Z_2^{2k})$.
Then the equation
$$0 = \langle M_k u, M_k^Tv \rangle = \langle M_k u, v \rangle ( 1 + \langle A_k, u+ v \rangle )$$
 implies that if there is an edge going from $u$ to $v$, then $u$ and $v$ are in different classes of the partition. 
So the interlacement graph of $T_k$ is bipartite.

Now suppose that the interlacement graph of the tree $T_k$ is bipartite.
Then 
there exists a vector $A_k \in \Z_2^{2k}$ such that
$$
 \langle M_k u, v \rangle ( 1 + \langle A_k, u+ v \rangle ) = 0
$$
holds for any $u, v \in \mathcal B (\Z_2^{2k})$.
Since 
$\langle M_k u, M_k^Tv \rangle = 0$
for any $u, v \in \mathcal B (\Z_2^{2k})$, 
the equation $$\langle M_k u, M_k^Tv \rangle =  \langle M_k u, v \rangle ( 1 + \langle A_k, u+ v \rangle )$$
holds for any $u, v \in \mathcal B (\Z_2^{2k})$.
Hence  by Proposition~\ref{inductionlemma} (2),
the same type of equation holds for the  matrix $M_0$ and a vector $A_0 \in \Z_2^{2k}$ as well.
Then the  $k \x k$ upper left submatrix $L$ of $M_0$ and the vector $A$ obtained from $A_0$ by 
taking only the first $k$ coordinates satisfy 
$$\langle L u, L^Tv \rangle =  \langle L u, v \rangle ( 1 + \langle A, u+ v \rangle )$$
for any $u, v \in \mathcal B (\Z_2^{k})$.
\end{proof}

\section{Proof of Proposition~\ref{inductionlemma}}

At first we prove two lemmas.
Let $r \geq 1$ and 
let $\mathcal B ( \Z_2^r )$ denote the set of standard basis elements of $\Z^r_2$.
For a matrix $M \in \Z_2^{r \x r}$ let
$S(M) \in \Z_2^{r \x r}$ be the matrix defined by
$\langle S(M) x, y \rangle = \langle Mx, y \rangle \langle My, x \rangle$, where
$x, y \in \mathcal B ( \Z_2^r )$.
Clearly we have $S(M) = S(M)^T = S(M^T)$. 

Suppose the diagonal of $M$ is fully zero. 
Let $p, p' \in \mathcal B ( \Z_2^r )$ fixed such that 
$p \neq p'$.
For any $u \in \mathcal B ( \Z_2^r )$ 
such that $u \neq p'$ let
 $P(u) = \langle M u, p \rangle \left( (M p +p') + \langle u, M p \rangle u \right)$ and
 $P(u)^T = \langle M^T u, p \rangle \left( (M^T p +p') + \langle u, M^T p \rangle u \right)$.  
 (Compare these formulas with Proposition~\ref{switchneigh}.)
 Also suppose that
 \begin{equation}\label{seged}
 \langle M^T v, p \rangle \langle M u  ,   p' \rangle + \langle M u, p \rangle  \langle  M^T v , p'  \rangle =0
 \end{equation}
 for all $u, v \in \mathcal B ( \Z_2^r )$ and
 $M$ is such that 
\begin{equation}\label{seged1.5}
\langle M p, p' \rangle = \langle M p', p \rangle = 0.
\end{equation}
\begin{lem}\label{elsolemma}
Let $A \in \Z_2^{r}$ be a fixed vector.
If for any  $x, y \in \mathcal B ( \Z_2^r )$  we have 
$$\langle M x, M^T y \rangle
= \langle M x, y \rangle (1+ \langle A, x + y \rangle ),$$
then
 for any  $u, v \in \mathcal B ( \Z_2^r )$ such that $u, v \neq p'$, we have
\begin{multline}\label{elsomasodik}
\langle M u + P(u) , M^T v + {P(v)^T}  \rangle = 
\langle M u + P(u) , v \rangle  ( 1 +  \langle A + S(M)p, u  + v \rangle  ).
\end{multline}
\end{lem}
\begin{proof}
Let 
$S$ denote $S(M)$ for short.
The right hand side of (\ref{elsomasodik}) is equal to 
\begin{multline*}
\langle Mu, v \rangle ( 1 + \langle A, u +v \rangle  ) + 
\langle Mu, v \rangle  \langle Sp, u  + v \rangle  + 
\langle P(u), v \rangle  ( 1 +  \langle A + Sp, u  + v \rangle  ).
\end{multline*}

By subtracting $\langle M u, M^T v \rangle =\langle Mu, v \rangle ( 1 + \langle A, u +v \rangle )$ from both sides 
of (\ref{elsomasodik}),
it is enough to show that
\begin{multline}\label{masodiklepes}
\langle M u  ,  {P^T(v)}  \rangle + \langle P(u) ,  M^T  v    \rangle
+ \langle P(u), P(v)^T \rangle = \\
\langle Mu, v \rangle  \langle Sp, u  + v \rangle  + 
\langle P(u), v \rangle  ( 1 +  \langle A + Sp, u  + v \rangle  ).
\end{multline}
By the definition of ${P(v)^T}$, we have
\begin{multline*}
\langle Mu  ,  {P(v)^T}  \rangle = 
\left \langle M u  , \langle  M^T v, p \rangle \left( ( M^T p +  p' )+ \langle  v, M^T p \rangle v \right) \right \rangle = \\
\langle M^T v, p \rangle \langle M u  , M^T p +  p' \rangle + \langle  Sp, v \rangle \langle M u  , v \rangle
\end{multline*}
and similarly
$$\langle P(u) , M^T v    \rangle = 
\langle M u, p \rangle \langle M p +  p', M^T v   \rangle + \langle Sp, u \rangle \langle  u,M^T v  \rangle.$$

Since 
by (\ref{seged}) $\langle M^T v, p \rangle \langle M u  ,   p' \rangle + \langle M u, p \rangle  \langle  p', M^T v   \rangle =0$,
we have
\begin{multline*}
\langle M u  ,  {P(v)^T}  \rangle + \langle P(u) , M^T v    \rangle =  \\
\langle  M^T v, p \rangle \langle M u  ,  M^T p  \rangle + \langle M u, p \rangle \langle  M p, M^T v  \rangle + 
 \langle  Sp, v \rangle \langle M u  , v \rangle +\langle Sp, u \rangle \langle  u,M^T v  \rangle = \\ 
 \langle M^T v, p \rangle  \langle M u  , p \rangle (1 +  \langle A, u +p \rangle )+ 
  \langle M u, p \rangle  \langle p, M^T v   \rangle (1 +  \langle A, p +v\rangle)  + \\
   \langle  Sp, v \rangle \langle M u  , v \rangle +\langle Sp, u \rangle \langle  u,M^T v  \rangle = \\
     \langle M u, p \rangle  \langle M^T v , p   \rangle  \langle A, u +v \rangle  +  \langle M u,  v  \rangle \langle Sp, u +v\rangle.
\end{multline*}

Now we have two cases.
\begin{enumerate}
\item
Suppose that $u \neq v$.
Then, we have 
$$
\langle  P(u) , v \rangle = \langle  \langle Mu, p \rangle \left( ( M p +  p' )+ \langle  u, M p \rangle u \right), v \rangle =\langle M u, p \rangle  \langle M p , 
v \rangle,
$$
since $v \neq p'$ and $u \neq v$.
So after putting the above results into  (\ref{masodiklepes}) and simplifying, we have to prove that
\begin{multline*}
\langle M u,  v  \rangle \langle Sp, u +v\rangle+ \langle  P(u) , {P(v)}^T\rangle  = \\
\langle M u,  v  \rangle \langle Sp, u +v\rangle + 
\langle M u, p \rangle  \langle M p , v \rangle (1+   \langle Sp, u +v\rangle),
\end{multline*}
which is equivalent to 
$$
\langle  P(u) , {P(v)}^T\rangle  = 
\langle M u, p \rangle  \langle M p , v \rangle (1+  \langle Sp, u +v\rangle ).
$$
Furthermore, we have
\begin{multline*}
\langle  P(u) ,   {P(v)}^T  \rangle = \\ 
\left \langle  \langle M u, p \rangle \left( ( M p +  p' )+ \langle  u, M p \rangle u \right),
 \langle   M^T v, p \rangle \left( (    M^T p +  p' )+ \langle  v,  M^T p \rangle v \right) \right \rangle = \\
 \langle M u, p \rangle   \langle   M^T v, p \rangle 
 \left( \langle p', p' \rangle +
  \langle M p,  M^T p \rangle + 
   \langle  v,   M^T p \rangle \langle M p, v \rangle + 
  \langle u, M p \rangle \langle u,  M^T p \rangle \right)  = \\
  \langle M u, p \rangle   \langle   M^T v, p \rangle + 0 +
   \langle M u, p \rangle   \langle   M^T v, p \rangle \left(  \langle   M^T v, p \rangle \langle  v,  M^T p \rangle  +  \langle M u, p \rangle \langle u, M p \rangle \right),
\end{multline*}
where we used (\ref{seged1.5}) and that $\langle M p,  M^T p \rangle = 0$.
But this is equal to 
$$
 \langle M u, p \rangle   \langle   M^T v, p \rangle ( 1 + \langle Sp, u\rangle + \langle Sp,v\rangle).
$$
This finishes the proof for the case $u \neq v$.
\item
Now suppose that $u = v$.

Then 
\begin{multline*}
\langle  P(u) , v \rangle =  \langle  P(u) , u \rangle = 
\langle  \langle Mu, p \rangle \left( ( M p +  p' )+ \langle  u, M p \rangle u \right), u \rangle = 2 \langle Mu, p \rangle   \langle Mp, u \rangle = 0.
\end{multline*}
Also
$$\langle M u  ,  {P(v)^T}  \rangle + \langle P(u) , M^T v    \rangle = 
 \langle M u, p \rangle  \langle M^T u , p   \rangle  \langle A, 2u \rangle  +  \langle M u,  u  \rangle \langle Sp, 2u\rangle = 0.$$
So we have to prove that  
$$
\langle  P(u) , {P(u)}^T\rangle  = 0.
$$
For this, we have 
\begin{multline*}
\langle  P(u) , {P(u)}^T\rangle  = \\
\langle  \langle M u, p \rangle \left( ( M p +  p' )+ \langle  u, M p \rangle u \right),
 \langle   M^T u, p \rangle \left( (    M^T p +  p' )+ \langle  u,  M^T p \rangle u \right) \rangle = \\
 2  \langle Mu, p \rangle  \langle Mp, u \rangle + 0 = 0.
\end{multline*}
Which again finishes the proof for the case $u =v \neq p'$.
\end{enumerate}
\end{proof}

Now we show the similar results for the cases when at least one of $u, v$ is equal to $p'$.  
Again suppose $M$ is such that 
\begin{equation}\label{seged2}
\langle M p, p' \rangle = \langle M p', p \rangle = 0.
\end{equation}
The necessary statements are the following.
\begin{lem}\label{masodiklemma}
Let $A \in \Z_2^{r}$ be a fixed vector.
If for any  $x, y \in \mathcal B ( \Z_2^r )$  we have 
$$\langle M x, M^T y \rangle
= \langle M x, y \rangle (1+ \langle A, x + y \rangle),$$
then
\begin{enumerate}[\rm (1)]
\item
we have for all $u \in \mathcal B ( \Z_2^r )$, $u \neq p'$
\begin{enumerate}[\rm (a)]
\item
\begin{multline}\label{masodik}
\langle M u + P(u) , M^T p  \rangle = 
\langle M u + P(u) , p' \rangle \left( 1 +  \langle A + S(M)p + \langle A, p \rangle p', u  + p' \rangle  \right) 
\end{multline}
if $\langle M u  , p' \rangle = 0$ and $\langle A, p' \rangle = 0$, and
\begin{equation}\label{nulla}
\langle M u + P(u) , 0  \rangle = \langle M u + P(u) , p' \rangle \left( 1 +  \langle A + S(M)p + \langle A, p \rangle p', u  + p' \rangle  \right)
\end{equation}
if $\langle M u  , p' \rangle = \langle M u  , p \rangle$,
\item
\begin{multline}\label{masodik2}
\langle M p , M^T u + P^T(u)  \rangle = 
\langle M p , u \rangle  \left( 1 +  \langle A + S(M)p + \langle A, p \rangle p', p'  + u \rangle \right),
\end{multline}
if 
$\langle A, p' \rangle = 0$, and clearly
\begin{equation}\label{nulla2}
\langle 0 , M^T u + P^T(u)  \rangle = 
\langle 0 , u \rangle  \left( 1 +  \langle A + S(M)p + \langle A, p \rangle p', p'  + u \rangle \right),\end{equation}
and
\end{enumerate}
\item
we have 
\begin{equation}\label{masodik3}
\langle M p , M^T p  \rangle = 
\langle M p , p' \rangle  ( 1 +  \langle A + S(M)p + \langle A, p \rangle p', p' + p' \rangle  ).
\end{equation}
and obviously
$
\langle 0 , 0  \rangle = 
\langle 0 , p' \rangle  ( 1 +  \langle A + S(M)p + \langle A, p \rangle p', p' + p' \rangle  ).
$
\end{enumerate}
\end{lem}
\begin{proof}
Let $S$ denote $S(M)$ for short.
For (1)(a) let us try to show
\begin{multline*}
\langle M u + \langle M u, p \rangle p'+ 
\langle  M u, p \rangle M p + \langle Sp, u\rangle u, 
 M^T p \rangle = \\
\langle M u + \langle M u, p \rangle p'+ 
\langle  M u, p \rangle M p + \langle Sp, u\rangle u,  p' \rangle 
 ( 1 +  \langle A + Sp + \langle A, p \rangle p', u  + p' \rangle  ).
\end{multline*}
The right hand side is equal to 
$ \langle Mu, p + p' \rangle ( 1 +  \langle A + Sp + \langle A, p \rangle p', u  + p' \rangle  )$ by (\ref{seged2}).
So (\ref{nulla}) is obvious. To show (\ref{masodik}) suppose $\langle M u  , p' \rangle = 0$ and $\langle A, p' \rangle = 0$.
Then we have to prove 
$$
\langle M u, M^T p \rangle + \langle Mu, p \rangle \langle M p , M^T p \rangle + \langle Sp, u \rangle = 
 \langle Mu, p  \rangle ( 1 +  \langle A + Sp + \langle A, p \rangle p', u  + p' \rangle  ),
$$
where we used (\ref{seged2}) and $\langle Sp, u\rangle \langle  u, M^T p \rangle = \langle Sp, u\rangle$.
Since $\langle M p , M^T p \rangle = 0$, this is further equivalent to
$$
 \langle Mu, p \rangle ( 1 + \langle A, u + p \rangle) +  \langle Sp, u \rangle =  \langle Mu, p \rangle ( 1 +  \langle A + Sp+ \langle A, p \rangle p', u  + p' \rangle  )
$$
which holds if $\langle A, p \rangle = \langle A+ \langle A, p \rangle p', p' \rangle$.
Since $\langle A, p' \rangle = 0$, we get the statement.

For (1)(b) since (\ref{nulla2})  obviously holds,  suppose 
$\langle A, p' \rangle = 0$. 
We have to show that
\begin{multline*}
\langle M p, 
 M^T u + \langle  M^T u, p \rangle p'+
\langle M^T u, p \rangle  M^T p + \langle Sp, u \rangle u  \rangle = \\
\langle M p, u \rangle 
 ( 1 +  \langle A + Sp+ \langle A, p \rangle p', p'  + u \rangle  ).
\end{multline*}
This is equivalent to
$$
\langle Mp, M^T u \rangle + \langle Sp, u \rangle = \langle M p, u \rangle 
 ( 1 +  \langle A + Sp+\langle A, p \rangle p', p'  + u \rangle  )
$$
which is further equivalent to
$$
\langle M p, u \rangle ( 1 +  \langle A , p  + u \rangle  )+ \langle Sp, u \rangle = \langle M p, u \rangle 
 ( 1 +  \langle A + Sp+\langle A, p \rangle p', p'  + u \rangle  )
$$
which holds similarly to the previous case.

Finally (2) holds because both sides are equal to zero.
\end{proof}

Now we are ready to prove  Proposition~\ref{inductionlemma}.
\begin{proof}[Proof of Proposition~\ref{inductionlemma}]
Let $A \in \Z_2^{2k}$ be a fixed vector
such that its $k+i+1, \ldots, 2k$-th coordinates are all zero.  
Assume that for any standard basis elements $x, y \in \Z_2^{2k}$  we have 
$$\langle M_i x, M^T_i y \rangle
= \langle M_{i} x, y \rangle (1+ \langle A, x + y \rangle).$$
In order to prove (1) we want to show that 
 there exists a  vector $B \in  \Z_2^{2k}$ such that
we have for any standard basis elements $u, v \in \Z_2^{2k}$
$$
\langle M_{i+1} u  , M^T_{i+1} v \rangle = 
\langle M_{i+1} u , v \rangle  ( 1 +  \langle B, u  + v \rangle  ).
$$
For the reader's convenience we recall that
$M_i$ denotes the incidence matrix of the paired tree $T_i$ and
$M_{i+1}$ is obtained from $M_i$ according to the formulas of Proposition~\ref{switchneigh}.

Note that with the roles $M = M_i$, $p = p_{i+1}$ and $p' = q_{i+1}$
$$\langle M^T v, p \rangle \langle M u  ,   p' \rangle + \langle M u, p \rangle  \langle  M^T v , p'  \rangle =0$$ holds
because
$\langle M u  ,   p' \rangle =  \langle  M^T v , p'  \rangle = 0$ for every $u, v \in \mathcal B (\Z_2^{2k})$
since in $M$  the row and column corresponding to $p'$ are entirely zero.
Let $$B = A + S(M_i)p_{i+1} + \langle A, p_{i+1} \rangle q_{i+1}.$$
If $u, v \neq q_{i+1}$, then
$$\langle B, u  + v \rangle = \langle A + S(M_i)p_{i+1} , u  + v \rangle$$ and 
by using Lemma~\ref{elsolemma} 
we get our claim (notice that (\ref{seged1.5}) holds).
And since (\ref{seged2}) holds and also $\langle A, p' \rangle = 0$, if at least one of $u, v$ is equal to $q_{i+1}$, then
by (\ref{masodik}), (\ref{masodik2}) in (1) of Lemma~\ref{masodiklemma} 
and by  (\ref{masodik3}) in (2) of Lemma~\ref{masodiklemma} 
 we get our claim.
 Notice that the $k+i+2, \ldots, 2k$-th coordinates of $B$ are all zero.  

To prove (2), we proceed similarly. Take $M = M_i$, $p = p_{i}$ and $p' = q_{i}$.
In this case 
$$\langle M^T v, p \rangle \langle M u  ,   p' \rangle + \langle M u, p \rangle  \langle  M^T v , p'  \rangle =0$$ holds
because $\langle M u  ,   p' \rangle =  \langle M u  ,   p \rangle$ 
and
$\langle  M^T v , p'  \rangle = \langle  M^T v , p  \rangle$
for every $u, v \in \mathcal B (\Z_2^{2k})$
since the rows and columns in $M$, respectively, corresponding to $p$ and $p'$ are the same. 

Let $$C = A + S(M_i)p_{i} + \langle A, p_{i} \rangle q_{i}.$$
If $u, v \neq q_{i}$, then
$$\langle C, u  + v \rangle = \langle A + S(M_i)p_{i} , u  + v \rangle$$ and 
by using Lemma~\ref{elsolemma}  again
we get our claim.
And if at least one of $u, v$ is equal to $q_i$, then by (\ref{nulla}), (\ref{nulla2}) and (2) of Lemma~\ref{masodiklemma} 
we get our claim.
\end{proof}

\end{spacing}

\begin{thebibliography}{AAAAAA}

\bibitem[Bor07]{Bo07}
A. Borb\'ely,
On the self-intersections of an immersed sphere,
Bull. Austral. Math. Soc. {\bf 75} (2007), 453--458.


\bibitem[Bou87]{Bou87}
A. Bouchet, Digraph decompositions and eulerian systems, SIAM J. Algebraic Discrete Methods {\bf 8} (1987), 323--337.


\bibitem[Bou94]{Bou94}
A. Bouchet, Circle graph obstructions, Journal of Combinatorial Theory, Series B {\bf 60} (1994), 107--144.


\bibitem[CKS04]{CKS04}
S. Carter, S. Kamada, M. Saito,
Surfaces in $4$-space, Springer, 2004.

\bibitem[De36]{De36}
M. Dehn,
\"Uber kombinatorische Topologie, Acta Math. {\bf 67} (1936), 123--168.


\bibitem[FO99]{FO99}
H. de Fraysseix and P. Ossona de Mendez,
On a characterization of Gauss codes,
Discrete Compute. Geom. {\bf 22} (1999), 287--295.


\bibitem[GG73]{GG}
M. Golubitsky and V. Guillemin, Stable mappings and their singularities, Graduate Texts in 
Math. {\bf 14}, Springer-Verlag, New York, 1973. 

\bibitem[Ka99]{Ka99}
L. H. Kauffman,  Virtual knot theory,  European Journal of Combinatorics 20 (7) (1999), 663--690.


\bibitem[Li04]{Li04}
G. Lippner, 
On double points of immersions of spheres, Man. Math. 113/2 (2004), 239--250.

\bibitem[LM76]{LM76}
L. Lov\'asz and M. L. Marx, A forbidden subgraph characterization of Gauss codes, Bull. Amer. Math. Soc. {\bf 82} (1976), 121--122. 



\bibitem[RR78]{RR78}
R. C. Read, P. Rosenstiehl,
On the Gauss crossing problem, Combinatorics (Proc. Fifth Hungarian Colloq., Keszthely, 1976), Vol. II, 
Colloq. Math. Soc. J\'anos Bolyai, 18, North-Holland, Amsterdam-New York, 1978, 843--876.


\end{thebibliography}
\end{document}